\documentclass[reqno]{amsproc}
\usepackage[utf8]{inputenc}
\usepackage{enumitem}
\usepackage{amsfonts}
\usepackage{orcidlink}
\usepackage{graphicx}
\usepackage{amscd}
\usepackage{amsmath}
\usepackage{amssymb}
\usepackage{latexsym}
\usepackage{hyperref}
\usepackage[all]{xy}
\usepackage{color}
\usepackage{mathrsfs}
\theoremstyle{plain}

\newtheorem{lemma}{\bf Lemma}
\newtheorem*{lemma*}{\bf Lemma}

\newtheorem{remark}{\bf Remark}

\newtheorem{theorem}{\bf Theorem}
\newtheorem{example}{\bf Example}
\numberwithin{equation}{section}
\usepackage{xcolor}
\DeclareMathOperator{\ric}{Ric}
\allowdisplaybreaks[4]
\usepackage{lineno}

\begin{document}

\title[On a special class of gradient Ricci solitons]{On a special class of gradient Ricci solitons}
\author[José N.V. Gomes]{José N.V. Gomes\orcidlink{0000-0001-5678-4789}}
\address{Departamento de Matemática, Universidade Federal de São Carlos, Rod. Washington Luís, Km 235, 13565-905, São Carlos, São Paulo, Brazil}
\email{jnvgomes@ufscar.br}
\urladdr{https://www.ufscar.br}

\author[Marcus A.M. Marrocos]{Marcus A.M. Marrocos\orcidlink{0000-0001-6336-0693}} 
\address{Departamento de Matemática, Universidade Federal do Amazonas, Av. General Rodrigo Octavio Jordão Ramos, 1200, 69080-900, Manaus, Amazonas, Brazil}
\email{marcusmarrocos@ufam.edu.br}
\urladdr{https://ufam.edu.br}


\begin{abstract}
We develop a method for constructing complete gradient Ricci solitons realized as fiber bundles endowed with warped metrics, and we establish necessary and sufficient conditions for their existence. As an application, we present new examples of complete gradient steady and shrinking Ricci solitons obtained via quotients by isometric group actions.
\end{abstract}
\maketitle

\section{Introduction}\label{intro}
A \emph{gradient Ricci soliton} is a triple $(M,g,\psi)$, where $(M,g)$ is a Riemannian manifold and $\psi$ is a smooth function on $M$ such that the Ricci tensor ${\rm Ric}_g$ of $g$ satisfies
\begin{equation}\label{eq:gradient-ricci-soliton}
{\rm Ric}_g+\nabla^2\psi = \lambda g
\end{equation}
for some constant $\lambda\in\mathbb{R}$. In this setting, $\psi$ is called the \emph{potential function}, and a gradient Ricci soliton is called \emph{shrinking}, \emph{steady}, or \emph{expanding} according as $\lambda>0$, $\lambda=0$, or $\lambda<0$. As is well known, up to diffeomorphisms and scaling, gradient Ricci solitons are special solutions to the Ricci flow that serve as geometric models describing the structure of its singularities. We refer the reader to the survey articles by Cao~\cite{Cao2010}, Cao, Chen, and Zhu~\cite{CaoChenZhu}, and Cao and Tran~\cite{CaoTran} for further details. 

Within this framework, the search for new examples remains a challenging problem. In the present work, we highlight some constructions that are particularly relevant to our interests. To the best of our knowledge, the most recent steady example was obtained by Lai~\cite{YiLai}, who constructed a family of three-dimensional complete gradient steady Ricci solitons known as flying wings, thereby confirming a conjecture of Richard S. Hamilton. In the Kähler setting, Bamler et al.~\cite{BCCD} constructed a new example of a complete gradient shrinking Kähler Ricci soliton with bounded scalar curvature and subsequently completed the classification of such solitons in complex dimension two. Very recently, a generalization of the flying wings solitons was announced by Chan, Lai, and Lee~\cite{ChanLaiLee}. Moreover, Bamler et al.~\cite{BCC} have also announced the existence of four-dimensional non-Kähler expanding Ricci solitons.

In this note, we introduce a method for constructing complete gradient Ricci solitons as fiber bundles endowed with warped metrics, which we call \emph{gradient Ricci soliton warped flat bundles}. Our construction is inspired by earlier examples of gradient Ricci solitons with warped product structures, whose prototype stems from a result of Bryant (although unpublished by him, it is discussed in~\cite{ChowEtal}). For instance, Ivey~\cite{Ivey} obtained a one-parameter family of complete noncompact gradient steady Ricci solitons of doubly warped product type on $\mathbb{R}^{k+1}\times N$, where $N$ is a compact Einstein manifold with positive scalar curvature. More recently, Angenent and Knopf~\cite{AngenentKnopf} constructed complete gradient shrinking Ricci solitons within a doubly warped product setting. The framework we develop provides a natural approach to reinterpret and extend these examples in the context of warped flat bundles.

The technique we employ is inspired by the generalization of the notion of warped product to bundles, as done by Bishop and O'Neill~\cite[p.~29]{BO}. More precisely, they observed that, given Riemannian manifolds $(B,g_B)$ and $(F,g_F)$, considering the universal covering $\tilde B$ of $B$ as a $\pi_1(B)$-bundle over $B$, a homomorphism $h\colon\pi_1(B)\to {\rm Iso}(F)$ gives rise to a Riemannian manifold $M=(\tilde B\times F)/\pi_1(B)$ with a metric $g$ and flat fiber bundle structure $F\to M\stackrel{\pi}{\to} B$, whose bundle charts $U\times F\to \pi^{-1}(U)$ are isometries. So, taking a smooth warping function $f$ on $B$, one defines a warped metric $g_f$ on $M$ by
\begin{equation}\label{WarpedMetric}
{g_f}(X,Y):=g(\mathcal{H}(X),\mathcal{H}(Y))+(f\circ\pi)^2g(\mathcal{V}(X),\mathcal{V}(Y)), \quad X,Y\in\mathfrak X(M),
\end{equation}
where $\mathcal{H}(X)$ and $\mathcal{V}(X)$ stand for the horizontal and vertical parts of $X\in\mathfrak X(M)$, respectively.

More generally, consider a principal bundle $G\to P\to B$, a manifold $F$, and an integrable horizontal distribution $\mathcal{H}$. We would like to establish necessary and sufficient conditions for the existence of a 
complete warped metric $g_f$ on 
\begin{equation*}
M=M(f,\psi):=(P\times F)/G
\end{equation*}
such that $(M,g_f,\psi)$ is a complete gradient Ricci soliton for some smooth potential function $\psi$ on $M$. In this case, we 
call $M(f,\psi)$ a \emph{gradient Ricci soliton warped flat bundle}. For our purposes, the following result of Feitosa et al.~\cite[Prop.~3]{FFG} will play a fundamental role.

\begin{lemma*}[\cite{FFG}]
Let $(B,g_B)$ be a complete Riemannian manifold. Assume there exist smooth functions $f>0$ and $\varphi$ on $B$ satisfying
\begin{equation}\label{EqProp1.16}
\ric+\nabla^2\varphi=\lambda g_B+\frac{m}{f}\nabla^2f\quad\mbox{and}
\quad 2\lambda\varphi-|\nabla\varphi|^2+\Delta\varphi+\frac{m}{f}\nabla\varphi(f)=c,
\end{equation}
for some constants $\lambda, m, c\in\mathbb{R}$, with $m\neq0$. Then, the functions $f$ and $\varphi$ satisfy
\begin{equation}\label{EqProp1.16 2}
\lambda f^2+f\Delta f+(m-1)|\nabla f|^2-f\nabla\varphi(f)=\mu
\end{equation}
for some constant $\mu\in\mathbb{R}$.
\end{lemma*}

We are now in a position to state our existence result. 

\begin{theorem}\label{characterization Ricci warped}
Suppose that $M(f,\psi)=(P\times F)/G$ is a gradient Ricci soliton warped flat bundle obtained from the principal bundle $G\to P\to B$. Then $\psi$ is the lift of a smooth function $\varphi$ on $B$ to $M$, which satisfies both equalities in~\eqref{EqProp1.16} with $m={\rm dim}\, F.$ Moreover, the Ricci tensor of $(F,g_F)$ is given by $Ric_F=\mu g_F$, where $\mu$ is a constant satisfying~\eqref{EqProp1.16 2}. Conversely, suppose there are two smooth functions $\varphi$ and $f>0$ on a Riemannian manifold $(B,g_B)$ satisfying the equations in~\eqref{EqProp1.16} and an $m$-dimensional Riemannian manifold $(F,g_F)$ such that $Ric_F=\mu g_F$ with $\mu$ given by~\eqref{EqProp1.16 2}. Then, there exists a gradient Ricci soliton warped flat bundle $M(f,\psi)$ with base $B$ and fiber $F$, where $P=\tilde B$ and $G=\pi_1(B)$.
\end{theorem}

Theorem~\ref{characterization Ricci warped}, together with Vilms' work~\cite[Thms.~2.2 and 3.6]{Vilms}, yields the following result, which provides a more refined converse to Theorem~\ref{characterization Ricci warped} under suitable assumptions.

\begin{theorem}\label{RecThmCRW}
Let $(B,g_{B})$ and $(F,g_F)$ be  complete Riemannian manifolds. Suppose that $f>0$ and $\varphi$ are smooth functions on $B$ satisfying the equations in~\eqref{EqProp1.16}, and that the Ricci tensor $\ric_F$ of $(F,g_F)$ satisfies $\ric_F=\mu g_F$, where $\mu$ is the constant in~\eqref{EqProp1.16 2}. Under these conditions, the following assertions hold:
\begin{enumerate}[label=(\roman*)]
\item [\rm (i)]\label{(i)} Let $G\to P\to B$ be an integrable principal bundle with a horizontal distribution $\mathcal{H}$. Assume that there exists a nontrivial homomorphism $h: G\to Iso(g_F)$. Then $M(f,\psi)$ is a gradient Ricci soliton warped flat bundle having $\mathcal H$ as an integrable horizontal distribution, and the fibers are totally umbilical and  diffeomorphic to $F.$
\smallskip
\item [\rm (ii)] \label{(ii)} Assume that there exists a finite normal Riemannian covering $\beta: \hat{B} \to B$ with covering transformation group $G$. Then $(\hat{B} \times F)/G$, endowed with the warped metric and potential function induced by the quotient projection, is a gradient Ricci soliton with integrable horizontal distribution and totally umbilical fibers.
\smallskip 
\item[\rm (iii)]\label{(iii)} If there exists a closed subgroup $G\subset{\rm Iso}(B)$ acting freely on $F$, such that $f$ and $\varphi$ are both invariant by $G$, then $(F\times B)/G$, endowed with the metric and potential function induced by the quotient projection, is a gradient Ricci soliton.
\end{enumerate}
\end{theorem}

It is worth noting that the Ricci flow literature already contains examples of gradient Ricci solitons arising from product and quotient constructions (see, e.g., the survey article of Cao, Chen and Zhu~\cite{CaoChenZhu}). For instance, Naber~\cite{Naber} showed that any four-dimensional complete noncompact shrinking Ricci soliton with bounded nonnegative curvature operator is isometric to either $\mathbb{R}^4$, or a finite quotient of $\mathbb{S}^3\times\mathbb{R}$ or $\mathbb{S}^2\times\mathbb{R}^2$. More recently, Cheng and Zhou~\cite{ChengZhou}, together with the works of Petersen and Wylie~\cite{PetersenWylie} and Fernández-López and García-Río~\cite{LopezRio}, proved that a four-dimensional complete noncompact gradient shrinking Ricci soliton has constant scalar curvature if and only if it is rigid, i.e., it is isometric to either an Einstein manifold or a finite quotient of the Gaussian shrinking soliton $\mathbb{R}^4$, $\mathbb{S}^3 \times \mathbb{R}$ or $\mathbb{S}^2 \times \mathbb{R}^2$. In particular, the model $\mathbb{S}^2\times\mathbb{R}^2$ is especially close to the point of view of this note, since it is a product-type shrinking soliton with Einstein fiber and Euclidean factor.

From the point of view of our construction, the previous examples are directly related to item~(iii) of Theorem~\ref{RecThmCRW}, which allows one to start from a product or warped product soliton and then produce a new soliton on the quotient $(F\times B)/G$, provided that the warping and potential functions are $G$-invariant. Ivey's construction provides complete gradient steady Ricci soliton warped products with Einstein fibers and potential functions depending only on the base, and thus fit directly into the framework of item~(iii) of Theorem~\ref{RecThmCRW}. Likewise, the complete gradient shrinking Ricci solitons constructed by Angenent and Knopf~\cite[Thm.~B]{AngenentKnopf}, although originally formulated as doubly warped products, can be reinterpreted (by incorporating one of the spherical factors into the base) as warped products with Einstein fibers and radial potential functions, and therefore also fall within the scope of our construction. In the product case, such as the model $\mathbb{S}^2\times\mathbb{R}^2$ appearing in Naber's classification, this corresponds simply to the special case of a constant warping function. Hence, item~(iii) of Theorem~\ref{RecThmCRW} applies in at least two distinct directions: it yields twisted quotient versions of Ivey's steady soliton warped products, and it also produces twisted quotient versions of shrinking solitons of product or doubly warped product type, such as $\mathbb{S}^2\times\mathbb{R}^2$ and the example of Angenent and Knopf. We present a concrete example in the next section to illustrate the preceding discussion.

\section{Proofs of the theorems}
This section is devoted to the proofs. Before presenting them, we first provide the example announced earlier.

\begin{example}\label{exam01}
\textnormal{Let $g_{B}=dt^2+a(t)^2\,g_{\mathbb{S}^k}$ be a Riemannian metric on $B=\mathbb{R}^{k+1}$ so that the warped metric $g=g_{B}+b(t)^2\,g_{\mathbb{S}^m}$ is a complete gradient Ricci soliton on $\mathbb{R}^{k+1}\times \mathbb{S}^m$ with a potential function which is the lift of a smooth function $\varphi(t)$, where $a(t)>0$, $b(t)>0$ and $\varphi(t)$ are nonconstant smooth functions arising from the constructions of Angenent and Knopf~\cite{AngenentKnopf} or Ivey~\cite{Ivey}. Now, consider $G=\mathbb{Z}_p$, with $p\ge 2$, and assume that $\mathbb{Z}_p$ acts freely and isometrically on $\mathbb{S}^m$ (for instance, this holds for any $p$ when $m$ is odd via the standard Hopf action, and it holds for $p=2$ in any dimension via the antipodal map, see Hatcher's book~\cite[p.~144]{Hatcher} for more details). Next, choose any isometric action of $\mathbb{Z}_p$ on $B=\mathbb{R}^{k+1}$ preserving the radial coordinate $t$ (e.g., a rotation of order $p$ in $\mathbb{R}^{k+1}$), so, $f=b(t)$ and $\varphi=\varphi(t)$ are automatically $\mathbb{Z}_p$-invariant, and the diagonal action on $\mathbb{S}^m\times B$ is free. Hence, from item (iii) of Theorem~\ref{RecThmCRW}, the quotient manifold
\[(\mathbb{S}^m\times B)/\mathbb{Z}_p \;=\; \mathbb{S}^m\times_{\,\mathbb{Z}_p} B\]
is a complete gradient Ricci soliton with metric and potential function both induced by the quotient projection.}
\end{example}

\begin{remark}
\textnormal{Note that, since in Example~\ref{exam01} the functions $a(t)$, $b(t)$, and $\varphi(t)$ are given by the constructions of Ivey or Angenent and Knopf, it follows that $\mathbb{S}^m \times_{\,\mathbb{Z}_p} B$ is, respectively, a gradient steady Ricci soliton or a gradient shrinking Ricci soliton.}
\end{remark}

We now recall a fundamental result from the theory of warped products and prove a key lemma to be used in our arguments. To this end, we follow the notation and terminology of Bishop and O'Neill~\cite{BO}. We begin by describing the geometric features of the metric $g_f$ in~\eqref{WarpedMetric}, which will be denoted by the superscript “$^-$”.

\begin{lemma}[\cite{BO}]\label{conex}
Suppose $X$ and $Y$ are basic vector fields which are $\pi$-related to $\bar X$ and $\bar Y$, respectively, and suppose $U$ and $V$ are vertical vector fields tangent to the fiber $F_b$, with $b=\pi(p)$, the we have:
\begin{enumerate}
\item $\bar D_XY$ is the horizontal lift of $^B\!D_{\bar X}\bar Y$;\label{P1 Lemma 1}
\smallskip
\item $\mathcal{H}(\bar D_UV)=-\frac{\bar g(U,V)}{\tilde{f}}\bar\nabla\tilde{f}$ and $\mathcal{V}(\bar D_UV)=\,\! ^{F_b}\!D_UV$;\label{P4 Lemma 1}
\smallskip
\item $\overline\ric(X,Y)=\,^B\!\ric(X,Y)-\frac{m}{\tilde f}H^f(X,Y)$;\label{P1 Lemma 2}
\smallskip
\item $\overline\ric(X,U)=0$;\label{P2 Lemma 2}
\smallskip
\item $\overline \ric(U,V)=\ric_{F_b}(U,V)-\left[\tilde{f}\tilde{\Delta} f+(m-1)|\nabla \tilde{f}|^2\right]g_{F_b}(U,V)$,
where $\tilde{\Delta}f=(\Delta f)\circ \pi$.\label{P3 Lemma 2}
\end{enumerate}
\end{lemma}

The following lemma will be used in the proof of the main results.

\begin{lemma}\label{Pot-Func-GRSWB}
For any gradient Ricci soliton warped flat bundle $M(f,\psi)$, the potential function $\psi$ depends only on the base.
\end{lemma}

\begin{proof}
Since we are considering $(P\times F)/G$ endowed with an integrable horizontal distribution and complete metrics on the base $B$ and on the fiber $F$, with this latter being invariant under the structural group $G$, it follows from the Vilms' theorem (see \cite[Thm.~3.6]{Vilms}) that the bundle projection $\pi: (P\times F)/G \to B$ is a nontrivial totally geodesic submersion, i.e., for each geodesic $\gamma(t)$ in $(P\times F)/G$, one has that $\pi(\gamma(t))$ is also a geodesic in $B$. Hence, by another Vilms' theorem (see~\cite[Thm.~2.2]{Vilms}), $(M,f,\psi):=(P\times F)/G$ is locally isometric to $B\times_fF$. Let $\iota:\mathcal{U}\to\mathcal{V}$ be a bundle chart isometry between open neighborhoods $\mathcal{V}=\mathcal{V}_B\times \mathcal{V}_F\subset B\times_f F$ and $\mathcal{U}\subset M(f,\psi)$. Since $M(f,\psi)$ is a gradient Ricci soliton, we have that $\mathcal{V}$ admits a gradient Ricci soliton warped product structure with potential function $\psi\circ \iota^{-1}$. By Borges and Tenenblat’s result~\cite{TenenblatBorges}, the potential function $\psi\circ\iota^{-1}$ is a lift of a smooth function on $\mathcal{V}_B$ to $\mathcal{V}$. Notice that this is enough to guarantee that $\psi$ depends only on the base, i.e., for any vertical vector field $V$, one has $V(\psi)=0.$ Thus, the potential function $\psi$ is constant along the fiber $\pi^{-1}(b)$, for every $b\in B$. Hence, we can define a function $\varphi$ on $B$ by $\varphi(b)=\psi(u)$, where $u\in\pi^{-1}(b)$, so $\psi=\varphi\circ\pi$. This function is smooth, since for any smooth local section $s: B\rightarrow M(f,\psi)$ with $s(b)\in\pi^{-1}(b)$, we have $\varphi(b)=\psi\circ s (b)$.
\end{proof}

Now, we are ready to prove the main theorems of this note.

\subsection{Proof of Theorem~\ref{characterization Ricci warped}}
\begin{proof}
As in the proof of Lemma~\ref{Pot-Func-GRSWB}, a gradient Ricci soliton warped flat bundle $M(f,\psi)$ is locally isometric to $B\times_fF$ and the potential function $\psi$ must be the lift $\tilde\varphi$ of a smooth function $\varphi$ on $B$ to $M$. Thus, the warped product $(B\times_fF,\tilde\varphi)$ is a gradient Ricci soliton as well. Hence, we can apply~\cite[Propositions~1 and 2]{FFG} to conclude that the functions $f$ and $\varphi$ satisfy the equations in \eqref{EqProp1.16} on $B$ for some constant $\lambda$, moreover, the Ricci tensor of $g_F$ is given by $Ric_F=\mu g_F$, for some constant $\mu$ satisfying \eqref{EqProp1.16 2}.

Reciprocally, suppose $(B,g_B)$ and $(F,g_F)$ are given as in the theorem, then we can construct a gradient Ricci soliton $M(f,\psi)$. For it, consider a homomorphism $h:\pi_1(B)\rightarrow Iso(F)$, and $\tilde B$ the $\pi_1(B)$-principal bundle over $B$, so that we have $M=(\tilde B\times F)/{\pi_1(B)}$. Now, take the Riemannian metric $g_f$ on $M$ as in~\eqref{WarpedMetric}, and observe that from $H^{\varphi}=\nabla^2\tilde{\varphi}$ and $H^f=\nabla^2\tilde{f}$ along horizontal vector fields, part~\eqref{P1 Lemma 2} of Lemma~\ref{conex} and the first equation of \eqref{EqProp1.16}, the Ricci soliton equation~\eqref{intro} is satisfied on the horizontal distribution. For mixed vector fields, when $X\in\mathcal{H}$ and $U\in\mathcal{V}$, we use $\bar\nabla\tilde{\varphi}\in\mathcal{H}$ and part~\eqref{P1 Lemma 1} of Lemma~\ref{conex} to obtain $\bar\nabla^2\tilde{\varphi}(X,U)=0$. So, by part~\eqref{P2 Lemma 2} of Lemma~\ref{conex}, the Ricci soliton equation is trivially satisfied. Now, for $U,V\in\mathcal{V}$, we have by definition of $\mu$ and part~\eqref{P3 Lemma 2} of Lemma~\ref{conex} that
\begin{equation*}
\overline{\ric}(U,V)=\left(\lambda-\frac{1}{f}\nabla\varphi(f)\right)g(U,V).
\end{equation*}
On the other hand, from part~\eqref{P4 Lemma 1} of Lemma~\ref{conex} we obtain
\begin{equation*}
\bar\nabla^2\tilde{\varphi}(U,V)=\frac{1}{f}\nabla\varphi(f)g(U,V).
\end{equation*}
Combining these two equations, we conclude that the Ricci soliton equation is satisfied.
\end{proof}

\subsection{Proof of Theorem~\ref{RecThmCRW}}

\begin{proof}
\noindent{\bf Item~(i):} Given an integrable principal bundle $G\to P\to B$ with a  horizontal distribution $\mathcal{H}$, and let $h:G\to Iso(g_F)$ be a homomorphism. Consider the associated fiber bundle $(P\times F)/G$ endowed with the flat connection determined by $\mathcal{H}$. Then, by Vilms~\cite[Thm.~3.5]{Vilms} one obtains a unique Riemannian metric $g$ on $(P\times F)/G$ for which the projection $\pi: (P\times F)/G \to B$ is a Riemannian submersion with totally geodesic fibers. Since $\mathcal{H}$ is integrable, Thms.~2.2 and~3.6 of Vilms~\cite{Vilms} imply that  $g$ is locally a product metric on vertical and horizontal neighborhoods. By considering the function $f>0$ on $B$, we can further warp $g$ to obtain the metric $g_f$ as in \eqref{WarpedMetric}. Taking $f$ and $\varphi$ to be smooth functions on $B$ satisfying the equations in~\eqref{EqProp1.16}, the result follows as in the proof of Theorem~\ref{characterization Ricci warped}.

\smallskip
\noindent{\bf Item~(ii):}
Let $\beta: \hat{B} \to B$ be a finite normal Riemannian covering map with covering transformation group $G$, and let $h:G\to Iso(g_F)$ be a homomorphism. Considering the hypothesis of the theorem, we can use our Theorem~\ref{characterization Ricci warped} for constructing a gradient Ricci soliton warped product $(\hat{B}\times_{\hat f} F,\hat{\varphi})$ with $G$-invariant parameters obtained from $g_{\hat B}=\beta^*g_{B}$, $\hat f=f\circ \beta$ and $\hat{\varphi}=\varphi\circ\beta$. Thus, we can pass to the quotient to obtain the Riemannian covering map $\hat\pi:\hat{B}\times_{\hat f} F\to (M,\bar{g}_{\hat f})$, where $M=(\hat{B}\times F)/G$ and $\bar{g}_{\hat f}$ is the quotient metric, so that $\hat{B}\times_{\hat f} F$ and $(M,\bar{g}_{\hat f})$ are locally isometric. Moreover, since $\hat{\varphi}$ is $G$-invariant, there exists a smooth function $\bar{\varphi}$ defined on $M$, such that $\bar\varphi\circ\hat\pi=\hat{\varphi}$. Hence, we have that $(M,\bar{g}_{\hat f},\bar\varphi)$ is a gradient Ricci soliton. It follows that $(M,\bar{g}_{\hat f})\to B$ has totally umbilical fibers, and, by construction, it has integrable horizontal distribution (see~\cite{Taubes}). 

\smallskip
\noindent{\bf Item~(iii):}
Since $G$ acts freely on $F$ and by isometries on $B$ we can consider a homomorphism $h:G\to Iso(g_{B})$ to get a flat bundle $M=(B\times F)/G$, where  the  action of $G$ on $B\times F$ is given by $a(b,c)=(h(a)b, ca)$, $a\in G$. Once the $f$ and $\varphi$ are invariant by isometries of $g_{B}$, the warped metric $\overline{g}$ is $G$-invariant (considering the action just defined). We can pass to the quotient to get the manifold $M$ equipped with the induced Riemannian metric $\hat{g}$ and potential function $\hat{\varphi}$, which is a gradient Ricci soliton. To see this, note that $(B\times_f F,\varphi)$ is a gradient Ricci soliton warped product, which is locally isometric to $(M,\hat g)$, and the result follows.
\end{proof}

\begin{remark}
\textnormal{In the setting of Theorem~\ref{characterization Ricci warped}, if $(B,g_B)$ has a nontrivial homotopy group, then $(\tilde{B}\times F)/\pi_1(B)$ admits a gradient Ricci soliton warped product structure with a nontrivial topology. This provides a first step toward understanding the construction of gradient Ricci solitons realized as general Riemannian submersions. A nontrivial example in the context of Riemannian submersions was explicitly constructed by Cao~\cite{Cao} in the special case of Kähler metrics.}
\end{remark}

\section{Acknowledgments}
The authors thank Ernani Ribeiro Jr and Ronaldo Freire de Lima for the critical reading of the manuscript and for the valuable suggestions that significantly improved the writing and the presentation of the results of this work. They are also deeply grateful to Huai-Dong Cao for bringing the work of Angenent and Knopf~\cite{AngenentKnopf} to our attention.

\section{\bf Declarations}

\subsection{\bf Conflict of interest} 
The authors declare no conflict of interest.



\subsection{\bf Author contributions}
All authors wrote the main manuscript text and reviewed the manuscript.

\subsection{\bf Funding}
José N.V. Gomes has been partially supported by the Conselho Nacional de Desenvolvimento Científico e Tecnológico (CNPq), Grant 310458/2021-8, and by the Fundação de Amparo à Pesquisa do Estado de São Paulo (FAPESP), Grants 2022/16097-2, 2023/11126-7 and 2024/00923-6. Marcus A.M. Marrocos was also partially supported by the latter agency, Grant 2020/14075-6.

\subsection{\bf Data Availability}
Data sharing was not applicable to this article as no datasets were generated or analyzed during the current study.

\end{document}